\documentclass[11pt,reqno]{amsart}
\usepackage{amsmath, amssymb, amsthm}
\usepackage{url}
\usepackage[breaklinks]{hyperref}
\usepackage[czech, english]{babel}

 \renewcommand{\a}{\alpha}
\renewcommand{\b}{\beta}

\newcommand{\g}{\gamma}
\newcommand{\G}{\Gamma}
\renewcommand{\l}{\lambda}

\renewcommand{\(}{\left\(}
\renewcommand{\)}{\right\)}
\renewcommand{\[}{\left\[}
\renewcommand{\]}{\right\]}

\numberwithin{equation}{section}
 \theoremstyle{plain}
\newtheorem{theorem}{Theorem}[section]
\newtheorem{lemma}[theorem]{Lemma}
\newtheorem{remark}[]{Remark}

\newtheorem{corollary}[theorem]{Corollary}

   \makeatletter
\def\proof{\@ifnextchar[{\@oproof}{\@nproof}}
\def\@oproof[#1][#2]{\trivlist\item[\hskip\labelsep\textit{#2 Proof of\
#1.}~]\ignorespaces}
\def\@nproof{\trivlist\item[\hskip\labelsep\textit{Proof.}~]\ignorespaces}

\makeatother

\begin{document}

\title[Generalized Lambert series and arithmetic nature of odd zeta values]{Generalized Lambert series and arithmetic nature of odd zeta values} 

\author{Atul Dixit and Bibekananda Maji}\thanks{2010 \textit{Mathematics Subject Classification.} Primary 11M06; Secondary 11J81.\\
\textit{Keywords and phrases.} Lambert series, Dedekind eta function, odd zeta values, Ramanujan's formula, transcendence}
\address{Discipline of Mathematics, Indian Institute of Technology Gandhinagar, Palaj, Gandhinagar 382355, Gujarat, India} 
\email{adixit@iitgn.ac.in, bibekananda.maji@iitgn.ac.in}


\begin{abstract}
It is pointed out that the generalized Lambert series $\displaystyle\sum_{n=1}^{\infty}\frac{n^{N-2h}}{e^{n^{N}x}-1}$ studied by Kanemitsu, Tanigawa and Yoshimoto can be found on page $332$ of Ramanujan's Lost Notebook in a slightly more general form. We extend an important transformation of this series obtained by Kanemitsu, Tanigawa and Yoshimoto by removing restrictions on the parameters $N$ and $h$ that they impose. From our extension we deduce a beautiful new generalization of Ramanujan's famous formula for odd zeta values which, for $N$ odd and $m>0$, gives a relation between $\zeta(2m+1)$ and $\zeta(2Nm+1)$. A result complementary to the aforementioned generalization is obtained for any even $N$ and $m\in\mathbb{Z}$. It generalizes a transformation of Wigert and can be regarded as a formula for $\zeta\left(2m+1-\frac{1}{N}\right)$. Applications of these transformations include a generalization of the transformation for the logarithm of Dedekind eta-function $\eta(z)$, Zudilin- and Rivoal-type results on transcendence of certain values, and a transcendence criterion for Euler's constant $\gamma$. 

\end{abstract}
\maketitle
\section{Introduction}\label{intro}
Surprises from Ramanujan's Lost Notebook \cite{lnb} seem to continue unabated. We recently came across a new instance of this phenomenon while reading page $332$ of the Lost Notebook, which we describe soon. Page $332$ is mainly devoted to two versions of a beautiful formula for $\zeta\left(\frac{1}{2}\right)$, one of which is also given in Entry 8 of Chapter $15$ of Ramanujan's second notebook \cite{ramnote}, \cite[p.~314]{bcbramsecnote}, and is as follows.

\textit{Let $\a$ and $\b$ be two positive numbers such that $\a\b=4\pi^3$. Then
\begin{align}\label{ramzetahalf}
\sum_{n=1}^{\infty}\frac{1}{e^{n^2\a}-1}=\frac{\pi^2}{6\a}+\frac{1}{4}+\frac{\sqrt{\b}}{4\pi}\left\{\zeta\left(\frac{1}{2}\right)+\sum_{n=1}^{\infty}\frac{\cos(\sqrt{n\b})-\sin(\sqrt{n\b})-e^{-\sqrt{n\b}}}{\sqrt{n}\left(\cosh(\sqrt{n\b})-\cos(\sqrt{n\b})\right)}\right\}.
\end{align}}
As mentioned in \cite[p.~191]{bcbramlost4}, one can conceive this result as an identity for an infinite sum of theta functions. Katsurada \cite{katsurada} interpreted the infinite series on the right as the exact form of the error term in the asymptotic formula
\begin{equation*}
\sum_{n=1}^{\infty}\frac{1}{e^{n^2\a}-1}=\frac{\pi^2}{6\a}+\frac{1}{4}+\frac{\sqrt{\b}}{4\pi}\zeta\left(\frac{1}{2}\right)+o(1)
\end{equation*}
as $\a\to 0^{+}$. The result \eqref{ramzetahalf} has instigated a flurry of activities over the past two decades with several mathematicians obtaining generalizations, including analogues for the Hurwitz zeta function and Dirichlet $L$-functions. A complete account of these activities is given in \cite[p.~191-193]{bcbramlost4} and while we avoid duplicating all references given there, we do mention those which are relevant to the material in the sequel. 

Ramanujan's second reformulation of \eqref{ramzetahalf} in \cite[p.~332]{lnb} appears in a footnote on page $9$ of a paper of Wigert \cite{wig} who also generalized it \cite[p.~8-9, Equation (5)]{wig} by deriving a formula for $\zeta\left(\frac{1}{N}\right)$ for any even positive integer $N$ without knowledge that the case $N=2$ had been considered by Ramanujan. This alternative reformulation of the formula for $\zeta\left(\frac{1}{2}\right)$ has been recently generalized in a different direction in \cite[p.~859, Theorem 10.1]{bdrz1}. Wigert's formula for $N$ even and $x$ positive\footnote{Wigert actually obtains the transformation for complex $x$ with Re$(x)>0$, however, this can be easily seen to be true by analytic continuation.} reads
\begin{align}\label{wigertN}
\sum_{n=1}^{\infty}\frac{1}{e^{n^{N}x}-1}&=\frac{\zeta(N)}{x}+x^{-\frac{1}{N}}\G\left(1+\frac{1}{N}\right)\zeta\left(\frac{1}{N}\right)+\frac{1}{4}\nonumber\\
&\quad+\frac{(-1)^{\frac{N}{2}-1}}{N}\left(\frac{2\pi}{x}\right)^{\frac{1}{N}}\sum_{j=0}^{\frac{N}{2}-1}\bigg\{e^{\frac{i\pi(2j+1)(N-1)}{2N}}\overline{L}_{N}\bigg(2\pi\left(\frac{2\pi}{x}\right)^{\frac{1}{N}}e^{-\frac{(2j+1)\pi i}{2N}}\bigg)\nonumber\\
&\qquad\qquad\qquad\qquad\qquad+e^{-\frac{i\pi(2j+1)(N-1)}{2N}}\overline{L}_{N}\bigg(2\pi\left(\frac{2\pi}{x}\right)^{\frac{1}{N}}e^{\frac{(2j+1)\pi i}{2N}}\bigg)\bigg\},
\end{align}
where $\overline{L}_{N}(x):=\sum_{n=1}^{\infty}\frac{n^{\frac{1}{N}-1}}{\text{exp}(n^{\frac{1}{N}}x)-1}.$

Kanemitsu, Tanigawa and Yoshimoto systematically studied various analogues and generalizations of \eqref{ramzetahalf} which give values of the Riemann zeta function \cite[Section 3]{ktyssfa}, Hurwitz zeta function \cite[Theorem 2.1]{ktyacta}, and also Dirichlet $L$-functions \cite[p.~32-33]{ktypic} at rational arguments with even denominator and odd numerator. The culmination of their results occurred in \cite[Theorem 1]{ktyhr} when they obtained the following beautiful result for values of the Riemann zeta function at positive rational arguments irrespective of the parity of the numerator and the denominator.

\textit{Let $h$ and $N$ be fixed natural numbers with $h\leq N/2$. Let 
\begin{equation}\label{abA}
a_{j,N}=\cos\left(\frac{\pi j}{2N}\right), b_{j,N}=\sin\left(\frac{\pi j}{2N}\right), A_{N}(y)=\pi\left(2\pi y\right)^{1/N}.
\end{equation}
Also, let
\begin{equation}\label{f0}
f_{0}(x;n,N)=\frac{e^{-A_{N}\left(\frac{n}{x}\right)}}{2\sinh\left(A_{N}\left(\frac{n}{x}\right)\right)}
\end{equation}
and for $j\geq 1$,
\begin{align}\label{f1}
f_{j}(x;n,N,h)=\frac{\cos\left(2A_{N}\left(\frac{n}{x}\right)b_{j,N}+\frac{\pi (2h-1)j}{2N}\right)-e^{-2A_{N}\left(\frac{n}{x}\right)a_{j,N}}\cos\left(\frac{\pi (2h-1)j}{2N}\right)}{\cosh\left(2A_{N}\left(\frac{n}{x}\right)a_{j,N}\right)-\cos\left(2A_{N}\left(\frac{n}{x}\right)b_{j,N}\right)}.
\end{align}
Then for $x>0$,
\begin{equation}\label{ramseries}
\sum_{n=1}^{\infty}\frac{n^{N-2h}}{e^{n^{N}x}-1}=P(x)+S(x),
\end{equation}
where
\begin{align}\label{pxfir}
P(x)=P(x;N,h)&=-\frac{1}{2}\zeta(-N+2h)+\frac{\zeta(2h)}{x}\nonumber\\
&\quad+\frac{1}{N}\G\left(\frac{N-2h+1}{N}\right)\zeta\left(\frac{N-2h+1}{N}\right)x^{-\frac{(N-2h+1)}{N}},
\end{align}
and\footnote{There is a minor typo in the statement of this theorem in \cite{ktyhr} in that the variable of summation in both the series representations for $S(x)$ there begins from $n=0$. See \eqref{sodd} and \eqref{seven} here for the corrected expressions.}
\begin{align}\label{sodd}
S(x)=S(x;N,h)=\frac{(-1)^{h+1}}{N}\left(\frac{2\pi}{x}\right)^{\frac{N-2h+1}{N}}\sum_{n=1}^{\infty}\frac{1}{n^{\frac{2h-1}{N}}}\bigg\{f_{0}(x;n,N)+\sum_{j=1}^{\frac{N-1}{2}}f_{2j}(x;n,N,h)\bigg\}
\end{align}
for $N$ odd, and 
\begin{align}\label{seven}
S(x)=S(x;N,h)=\frac{(-1)^{h+1}}{N}\left(\frac{2\pi}{x}\right)^{\frac{N-2h+1}{N}}\sum_{n=1}^{\infty}\frac{1}{n^{\frac{2h-1}{N}}}\sum_{j=1}^{\frac{N}{2}}f_{2j-1}(x;n,N,h)
\end{align}
for $N$ even.}
Note that Ramanujan's formula \eqref{ramzetahalf} is the special case $h=1, N=2$ of the above result.

Now the surprising thing that we came across while going over page $332$ of the Lost Notebook \cite{lnb} is that at the end of this page Ramanujan starts writing the \emph{exact} same series considered by Kanemitsu, Tanigawa and Yoshimoto in their above result, that is, the series on the left side of \eqref{ramseries}, but with more general conditions on the associated parameters subsuming the ones given by Kanemitsu et al. The precise sentence at the end of this page, in Ramanujan's own words, reads
\textit{\begin{equation}\label{ramexseries}
\frac{1^{r}}{e^{1^sx}-1}+\frac{2^{r}}{e^{2^sx}-1}+\frac{3^{r}}{e^{3^sx}-1}+\cdots
\end{equation}
where $s$ is a positive integer and $r-s$ is any even integer.}

Note that Ramanujan is taking $r$ and $s$ such that $r-s$ is \emph{any} even integer, where as for the series on the left side of \eqref{ramseries}, Kanemitsu et al., in Ramanujan's notation, take $r-s$ to be a \emph{negative} even integer only.

Ramanujan does not give any expression for this series like the one in \eqref{ramseries}, so it is not clear what he had in his mind. Was there a page after page $332$ in the Lost Notebook that went missing? While this question may remain unanswered forever and while some may call our speculation as wishful thinking, from the fact that the left-hand side of \eqref{ramzetahalf} is merely a special case of \eqref{ramexseries}, it is clear that Ramanujan had recognized the importance of the latter. Perhaps if he had completed his formula, he would have ended up with an expression consisting of the Riemann zeta function at rational arguments. 

Even though Ramanujan did not give any expression for the series in \eqref{ramexseries}, he is right in making the assumption that $r-s$ is any even integer, for, \eqref{ramseries} holds not only for $0<h\leq N/2$ but also when $h$ is \emph{any} integer. This, in Ramanujan's notation, is equivalent to saying $r-s$ can be any even integer. This observation has two fruitful consequences that went unnoticed in the paper \cite{ktyhr} of Kanemitsu et al. 

Indeed, if we allow $h$ to be any negative integer, then the case $N=1, h=1-m$, where $m>1$ is a natural number, of \eqref{ramseries} gives an interesting result of Ramanujan \cite[vol. 1, p.~259, no. 14]{ramnote}, \cite[p.~269]{trigsums}, \cite[p.~190]{ramcollected}, namely, for $\a,\b>0$ with $\a\b=\pi^2$, 
\begin{align}\label{ramzetaspl}
\a^{m}\sum_{n=1}^{\infty}\frac{n^{2m-1}}{e^{2\a n}-1}-(-\b)^m\sum_{n=1}^{\infty}\frac{n^{2m-1}}{e^{2\b n}-1}=\left(\a^m-(-\b)^m\right)\frac{B_{2m}}{4m},
\end{align}
where $B_{m}$ is the $m^{\textup{th}}$ Bernoulli number defined by
\begin{equation*}
\frac{x}{e^{x}-1}=\sum_{n=1}^{\infty}\frac{B_nx^n}{n!},\hspace{5mm}(|x|<2\pi),
\end{equation*}
whereas the case $N=1, h=0$ of \eqref{ramseries} gives another result of Ramanujan \cite[Ch. 14, Sec. 8, Cor. (i)]{ramnote}, \cite[p.~318, formula (23)]{lnb}, namely, for $\a,\b>0$ such that $\a\b=\pi^2$,
\begin{equation}\label{ramzetaspl0}
\a\sum_{n=1}^{\infty}\frac{n}{e^{2n\a}-1}+\b\sum_{n=1}^{\infty}\frac{n}{e^{2n\b}-1}=\frac{\a+\b}{24}-\frac{1}{4}.
\end{equation}
However, \eqref{ramzetaspl} and \eqref{ramzetaspl0} are but special cases of Ramanujan's following famous formula for $\zeta(2m+1)$ \cite[p.~173, Ch. 14, Entry 21(i)]{ramnote}, \cite[p.~319-320, formula (28)]{lnb}, \cite[p.~275-276]{bcbramsecnote}, when $m$ is a negative integer less than $-1$ and $m=-1$ respectively.

\textit{For $\a, \b>0$ with $\a\b=\pi^2$ and $m\in\mathbb{Z}, m\neq 0$,
\begin{align}\label{zetaodd}
\a^{-m}\left\{\frac{1}{2}\zeta(2m+1)+\sum_{n=1}^{\infty}\frac{n^{-2m-1}}{e^{2\a n}-1}\right\}&=(-\b)^{-m}\left\{\frac{1}{2}\zeta(2m+1)+\sum_{n=1}^{\infty}\frac{n^{-2m-1}}{e^{2\b n}-1}\right\}\nonumber\\
&-2^{2m}\sum_{j=0}^{m+1}\frac{(-1)^jB_{2j}B_{2m+2-2j}}{(2j)!(2m+2-2j)!}\a^{m+1-j}\b^j.
\end{align}}
As a special case, this formula gives the following result of Lerch \cite{lerch} for $m$ odd.
\begin{align}\label{lerch}
\zeta(2m+1)+2\sum_{n=1}^{\infty}\frac{1}{n^{2m+1}(e^{2\pi n}-1)}=\pi^{2m+1}2^{2m}\sum_{j=0}^{m+1}\frac{(-1)^{j+1}B_{2j}B_{2m+2-2j}}{(2j)!(2m+2-2j)!}.
\end{align} 
For the history, discussion and reference to works on Ramanujan's beautiful formula \eqref{zetaodd}, we refer the reader to \cite{berndtrocky} and to the more recent paper \cite{berndtstraubzeta}. A modern interpretation of this formula is that it encodes fundamental transformation properties of Eisenstein series on $\textup{SL}_{2}(\mathbb{Z})$ and their Eichler integrals \cite{gmr}. This interpretation has been extended in \cite[Section 5]{berndtstraubmathz} to weight $2k+1$ Eisenstein series of level $2$. Kirschenhofer and Prodinger \cite{kirprod} have found applications of Ramanujan's formula in the analysis of special data structures and algorithms. More specifically, they use these identities to achieve certain distribution results on random variables related to dynamic data structures called `tries', which are of importance in theoretical computer science.

In view of \eqref{ramzetaspl} and \eqref{ramzetaspl0} both resulting from \eqref{ramseries}, the natural question that arises now is - does \eqref{ramseries} also give \eqref{zetaodd} for the remaining integer values of $m$? The answer is no. This motivated us to look for an extension of \eqref{ramseries} for other values of $h$ that may answer the above question affirmatively. Our extension of \eqref{ramseries} for $N-2h\neq-1$ is given below. The remaining case when $N-2h=-1$ is dealt with separately in Theorem \ref{zetagenm0}.
\begin{theorem}\label{ktycomp}
Let $N\in\mathbb{N}$ and $h$ be any integer. Let $a_{j,N}, b_{j,N}$ and $A_{N}(y)$ be defined in \eqref{abA}, $f_{0}(x;n,N), f_{j}(x;n,N,h), j\geq 1$, be defined in \eqref{f0} and \eqref{f1}. Let $x>0$ and let $P(x)$ be defined in \eqref{pxfir} and $S(x)$ in \eqref{sodd} and \eqref{seven}. Then for $N-2h\neq-1$,
\begin{equation*}
\sum_{n=1}^{\infty}\frac{n^{N-2h}}{e^{n^{N}x}-1}=P_{1}(x)+S(x),
\end{equation*}
where
\begin{align}\label{p1x}
P_{1}(x)=P_{1}(x;N,h)=P(x)+(-1)^{h+1}2^{2h-1}\pi^{2h}\sum_{j=1}^{\left\lfloor\frac{h}{N}\right\rfloor}\left(\frac{-1}{4\pi^2}\right)^{jN}\frac{B_{2j}B_{2h-2jN}x^{2j-1}}{(2j)!(2h-2jN)!}.
\end{align}
\end{theorem} 
We note here that Kanemitsu, Tanigawa and Yoshimoto \cite[Theorem 2.1]{ktyacta} have obtained the above extension but only in the case when $N$ is even and $h\geq N/2$. They do not obtain the above extension when $N$ is odd. Nor do they obtain the result for $N$ even and $h<0$. But for $N$ even and $h\geq N/2$, they obtain a result not only for the Riemann zeta function but also for the multiple Hurwitz zeta function $\zeta_{k}(s,a)$. Later \cite[p.~32]{ktypic}, they also obtain a character analogue of this result, but again for $N$ even and $h\geq N/2$. (We note that the results in \cite[Theorem 2.1]{ktyacta} and \cite[p.~32]{ktypic} contain two more variables $\ell$ and $a$ apart from $h$ and $N$, however, for $a=1$, which is what corresponds to $\zeta(s)$, we can rephrase the conditions for the validity of their results in the form we have given, namely, $N$ even and $h\geq N/2$.) Katsurada \cite{katsurada} also obtains similar results but which do not contain any power of $n$.

Our extension in Theorem \ref{ktycomp}, on the other hand, allows us to have no restrictions on $h$ and $N$ except that $N$ be any natural number and $h$ be any integer. As we show in this paper, our Theorem \ref{ktycomp} for $N$ odd, and $h\leq 0$ or $h>N/2$, not only gives Ramanujan's famous formula \eqref{zetaodd} as a special case but also its new elegant generalization stated in Theorem \ref{zetagen} below. An important thing about this generalization is that it gives a relation between $\zeta(2m+1)$ and $\zeta(2Nm+1)$ for \emph{any} odd positive integer $N$ and \emph{any} non-zero integer $m$. Such a relation between these odd zeta values has been missing from the literature. 
\begin{theorem}\label{zetagen}
Let $N$ be an odd positive integer and $\a,\b>0$ such that $\a\b^{N}=\pi^{N+1}$. Then for any non-zero integer $m$,
{\allowdisplaybreaks\begin{align}\label{zetageneqn}
&\a^{-\frac{2Nm}{N+1}}\left(\frac{1}{2}\zeta(2Nm+1)+\sum_{n=1}^{\infty}\frac{n^{-2Nm-1}}{\textup{exp}\left((2n)^{N}\a\right)-1}\right)\nonumber\\
&=\left(-\b^{\frac{2N}{N+1}}\right)^{-m}\frac{2^{2m(N-1)}}{N}\Bigg(\frac{1}{2}\zeta(2m+1)+(-1)^{\frac{N+3}{2}}\sum_{j=\frac{-(N-1)}{2}}^{\frac{N-1}{2}}(-1)^{j}\sum_{n=1}^{\infty}\frac{n^{-2m-1}}{\textup{exp}\left((2n)^{\frac{1}{N}}\b e^{\frac{i\pi j}{N}}\right)-1}\Bigg)\nonumber\\
&\quad+(-1)^{m+\frac{N+3}{2}}2^{2Nm}\sum_{j=0}^{\left\lfloor\frac{N+1}{2N}+m\right\rfloor}\frac{(-1)^jB_{2j}B_{N+1+2N(m-j)}}{(2j)!(N+1+2N(m-j))!}\a^{\frac{2j}{N+1}}\b^{N+\frac{2N^2(m-j)}{N+1}}.
\end{align}}
\end{theorem}
It is easy to see that Ramanujan's formula for $\zeta(2m+1)$, that is \eqref{zetaodd}, is the special case $N=1$ of the above result. Thus Theorem \ref{zetagen} gives an infinite family of Ramanujan-type identities for odd zeta values.

The famous result of Euler, namely, 
\begin{equation}\label{zetaevenint}
\zeta(2 m ) = (-1)^{m +1} \frac{(2\pi)^{2 m}B_{2 m }}{2 (2 m)!}
\end{equation}
implies that for all $m\geq 1$, the even zeta values $\zeta(2m)$ are rational multiples of $\pi^{2m}$, and hence transcendental. However, the arithmetic nature of the odd zeta values $\zeta(2m+1)$ is mysterious. Till date we know, thanks to Ap\'{e}ry \cite{apery1}, \cite{apery2}, that $\zeta(3)$ is irrational. But we do not know whether $\zeta(3)$ is transcendental or not. Moreover, even though it is known, due to Rivoal \cite{rivoal} and Ball and Rivoal \cite{ballrivoal}, that there exist infinitely many odd zeta values $\zeta(2m+1), m\geq 1$, which are irrational, it is not known, as of yet, whether $\zeta(2m+1)$ is transcendental or irrational if $m$ is any specific natural number greater than or equal to $2$. Currently the best result in this area, due to Zudilin \cite{zudilin}, states that at least one of $\zeta(5), \zeta(7), \zeta(9)$ or $\zeta(11)$ is irrational. We note that very recently Han\v{c}l and Kristensen \cite{hanclkristensen} have obtained some criteria for irrationality of odd zeta values and Euler's constant.

As applications of the above theorem, we derive Zudilin- and Rivoal-type results on transcendence of odd zeta values and generalized Lambert series in Section \ref{app}.

We next give a one-parameter generalization of the following well-known transformation formula \cite[Ch.~14, Sec.~8, Cor.~(ii); Ch.~16, Entry 27(iii)]{ramnote}, \cite[p.~256]{bcbramsecnote}, \cite[p.~43]{bcbramthinote}, \cite[p.~320, Formula (29)]{lnb} for the logarithm of Dedekind eta function $\eta(z)$:

For $\a, \b>0$ and $\a\b=\pi^2$,
\begin{align}\label{logdede}
\sum_{n=1}^{\infty}\frac{1}{n(e^{2n\a}-1)}-\sum_{n=1}^{\infty}\frac{1}{n(e^{2n\b}-1)}=\frac{\b-\a}{12}+\frac{1}{4}\log\left(\frac{\a}{\b}\right).
\end{align}
This generalization corresponds to the special case $N-2h=-1$, that is, $h=\frac{N+1}{2}$, of the series $\displaystyle\sum_{n=1}^{\infty}\frac{n^{N-2h}}{e^{n^{N}x}-1}$ and is given below.
\begin{theorem}\label{zetagenm0}
Let $N$ be an odd positive integer and $\a,\b>0$ such that $\a\b^{N}=\pi^{N+1}$. Let $\gamma$ denote Euler's constant. Then
{\allowdisplaybreaks\begin{align}\label{zetagenm0eqn}
&\sum_{n=1}^{\infty}\frac{1}{n\left(\exp{((2n)^{N}\a)}-1\right)}-\frac{(-1)^{\frac{N+3}{2}}}{N}\sum_{j=\frac{-(N-1)}{2}}^{\frac{N-1}{2}}(-1)^{j}\sum_{n=1}^{\infty}\frac{1}{n\left(\textup{exp}\left((2n)^{\frac{1}{N}}\b e^{\frac{i\pi j}{N}}\right)-1\right)}\nonumber\\
&=\frac{(N-1)(\log 2-\gamma)}{2N}+\frac{\log(\a/\b)}{2(N+1)}+(-1)^{\frac{N+3}{2}}\sum_{j=0}^{\left\lfloor\frac{N+1}{2N}\right\rfloor}\frac{(-1)^jB_{2j}B_{N+1-2Nj}}{(2j)!(N+1-2Nj)!}\a^{\frac{2j}{N+1}}\b^{N-\frac{2N^2j}{N+1}}.
\end{align}}
\end{theorem}
\begin{remark}\label{ng1}
Note that 
\begin{equation*}
\sum_{j=0}^{\left\lfloor\frac{N+1}{2N}\right\rfloor}\frac{(-1)^jB_{2j}B_{N+1-2Nj}}{(2j)!(N+1-2Nj)!}\a^{\frac{2j}{N+1}}\b^{N-\frac{2N^2j}{N+1}}=
\begin{cases}
\frac{\b-\a}{12},\quad\hspace{6mm}\text{if}\hspace{1mm} N=1,\\
\frac{B_{N+1}\b^{N}}{(N+1)!}, \quad\text{if}\hspace{1mm} N>1.
\end{cases}
\end{equation*}
\end{remark}
It is easy to see that when $N=1$, the above theorem reduces to \eqref{logdede}. 

As mentioned before, Kanemitsu, Tanigawa and Yoshimoto \cite[Theorem 2.1]{ktyacta} obtained our extension in Theorem \ref{ktycomp}, but only in the case when $N$ is even and $h\geq N/2$. Our extension in Theorem \ref{ktycomp} not only covers the remaining case, that is, $N$ even and $h\leq 0$ but also the case when $N$ is odd and $h$ is any integer. Also, in Theorem \ref{ktycomp}, if we let $N$ be an even positive integer and if $h=\left(m+\frac{1}{2}\right)N$, where $m$ is \emph{any} integer, we obtain the following result which is complementary to Theorem \ref{zetagen}, and which generalizes Wigert's formula \eqref{wigertN}. 
\begin{theorem}\label{neven}
Let $N$ be an even positive integer and $m$ be \emph{any} integer. For any $\a,\b>0$ satisfying $\a\b^N=\pi^{N+1}$,
{\allowdisplaybreaks\begin{align}\label{neveneqn}
&\a^{-\left(\frac{2Nm-1}{N+1}\right)}\left(\frac{1}{2}\zeta(2Nm)+\sum_{n=1}^{\infty}\frac{n^{-2Nm}}{\exp{\left((2n)^{N}\a\right)}-1}\right)\nonumber\\
&=\b^{-\left(\frac{2Nm-1}{N+1}\right)}\frac{(-1)^m}{N}2^{(N-1)\left(2m-\frac{1}{N}\right)}\Bigg(\frac{\zeta\left(2m+1-\frac{1}{N}\right)}{2\cos\left(\frac{\pi}{2N}\right)}\nonumber\\
&\quad-2(-1)^{\frac{N}{2}}\sum_{j=0}^{\frac{N}{2}-1}(-1)^j \sum_{n=1}^{\infty}\frac{1}{n^{2m+1-\frac{1}{N}}}\textup{Im}\Bigg(\frac{e^{\frac{i\pi(2j+1)}{2N}}}{\exp{\left((2n)^{\frac{1}{N}}\b e^{\frac{i\pi(2j+1)}{2N}}\right)}-1}\Bigg)\Bigg)\nonumber\\
&\quad+(-1)^{\frac{N}{2}+1}2^{2Nm-1}\sum_{j=0}^{m}\frac{B_{2j}B_{(2m+1-2j)N}}{(2j)!((2m+1-2j)N)!}\a^{\frac{2j}{N+1}}\b^{N+\frac{2N^2(m-j)-N}{N+1}}.
\end{align}}
\end{theorem}
Wigert's formula \eqref{wigertN} for real $x>0$ can be proved from Theorem \ref{neven} simply by taking $m=0$ and $\a=x/2^N, \b=(\pi^{N+1}/\a)^{\frac{1}{N}}$ and simplifying the resultant. At one point in the proof, one also needs to use the functional equation for $\zeta(s)$ in the form
\begin{equation}\label{zetafe}
\zeta(s)=2^s\pi^{s-1}\G(1-s)\zeta(1-s)\sin\left(\tfrac{1}{2}\pi s\right).
\end{equation}
This formula of Wigert, which could have been derived just from \eqref{ramseries} in the special case $h=N/2$, was missed by the authors in \cite{ktyhr}. The reason we put it in the above form is to compare it with Theorem \ref{zetagen}. Letting $N=2$, in turn, in Wigert's formula leads us to \eqref{ramzetahalf} after redefining the variables $\a$ and $\b$ to satisfy the condition on them given by Ramanujan.

An application of Theorem \ref{neven} towards transcendence of certain values is discussed in Section \ref{app}.

This paper is organized as follows. In Section \ref{exten} we prove Theorem \ref{ktycomp}. Section \ref{oddd} is devoted to proving the generalization of Ramanujan's formula for odd zeta values, that is, Theorem \ref{zetagen}. We also prove Theorem \ref{zetagenm0} in this section. Theorem \ref{neven}, which is a generalization of Wigert's formula, is proved in Section \ref{evenn}. Section \ref{app} is devoted to proving interesting results on transcendence of certain values as a result of applications of our transformations obtained in the previous sections. In Section \ref{cnesec}, we point out an error in a result of Chandrasekharan and Narasimhan, which when corrected, results in nothing but Ramanujan's formula for $\zeta(2m+1)$ for $m>0$. We conclude the paper with some remarks and future directions in Section \ref{cr}.

\section{An extension of the Kanemitsu-Tanigawa-Yoshimoto theorem}\label{exten}
We prove Theorem \ref{ktycomp} here. Even though some details in the proof are the same as in the proof of \eqref{ramseries} given in \cite{ktyhr}, we repeat them here for the sake of completeness. 

Theorem \ref{ktycomp} in the case $0<h\leq N/2$ is already proved in \cite{ktyhr}. As remarked in the introduction, their proof extends to $h\leq 0$ without much further effort. Now assume $h>N/2$. For Re$(s)=c_0>\max\left(\frac{N-2h+1}{N},1\right)=1$, it is easy to see that
\begin{align}\label{integ}
\sum_{n=1}^{\infty}\frac{n^{N-2h}}{e^{n^{N}x}-1}=\frac{1}{2\pi i}\int_{(c_{0})}\G(s)\zeta(s)\zeta\left(Ns-\left(N-2h\right)\right)x^{-s}\, {\rm d}s.
\end{align}
Here, and throughout the sequel, $\int_{(c)}$ denotes the line integral $\int_{c-i\infty}^{c+i\infty}$. 

We now shift the line of integration from Re$(s)=c_0$ to Re$(s)=-c$, where $c>\frac{2h}{N}-1$. The reason to choose this lower bound for $c$ is now explained.

It is well-known that $\G(s)$ has simple poles at $s=0$ and at negative integers. The simple trivial zeros of $\zeta(s)$ cancel the poles of $\G(s)$ at all negative even integers. The pole of $\zeta(Ns-(N-2h))$ is at $s=\frac{N-2h+1}{N}$ where as the trivial zeros of $\zeta(Ns-(N-2h))$ are at $s=\frac{N-2h-2j}{N}, j\in\mathbb{N}$. The pole of $\G(s)$ at $s=0$ does not get canceled with any zero of $\zeta(Ns-(N-2h))$, for if $\frac{N-2h-2j}{N}=0$ for some $j\in\mathbb{N}$, then this implies $h=N/2-j<N/2$, which is a contradiction. Similarly the pole of $\zeta(s)$ at $s=1$ does not cancel with any zero of $\zeta(Ns-(N-2h))$. Since the trivial zeros of $\zeta(s)$ are simple and they are already used in the poles of $\G(s)$ at the negative even integers, the pole of $\zeta(Ns-(N-2h))$ at $s=\frac{N-2h+1}{N}$ does not get canceled.

We next find out which poles of $\G(s)$ at the negative odd integers get canceled with the trivial zeros of $\zeta(Ns-(N-2h))$. Suppose $j$ is the minimum positive integer such that the zero $(N-2h-2j)/N$ of $\zeta(Ns-(N-2h))$ cancels with the pole $s=-(2\ell-1)$ of $\G(s)$. Then note that $\ell$ will also be the minimum positive integer for which this occurs. Now we show that all the poles of $\G(s)$ at negative odd integers after $-(2\ell-1)$ also get canceled. Note that this would happen only if $(N-2h-2j')/N=-(2\ell+f)$ for $f$ an odd positive integer and $j'$ some natural number. This would then imply $2h+2j'=N(2\ell+f+1)$, which, indeed, is valid for some natural number $j'$. Thus, all of the poles $-(2\ell-1), -(2\ell+1), -(2\ell+3),\cdots$ of $\G(s)$ get canceled by the zeros of $\zeta(Ns-(N-2h))$.

So the only poles of $\G(s)$ which remain to be investigated are $-1, -3, -5,\cdots, -(2\ell-3)$. Since $\ell$ depends on $N$ and $h$, we now need to find an expression for it in terms of $h$ and $N$. According to the definition, $j$ (and hence $\ell$) is the minimum positive integer such that $(N-2h-2j)/N=-(2\ell-1)$, that is, $j=\ell N-h$. Since $j>0$, we must have $\ell>h/N$, and since $\ell$ is the minimum positive integer for which this occurs, we see that $\ell=\left\lfloor\frac{h}{N}\right\rfloor+1$. This implies that $-(2\ell-3)=-\left(2\left\lfloor\frac{h}{N}\right\rfloor-1\right)$. Thus, the only poles of $\G(s)$ which contribute are $-1, -3, -5,\cdots, -\left(2\left\lfloor\frac{h}{N}\right\rfloor-1\right)$. 

This suggests that we shift the line of integration from $c_0>1$ to Re$(s)=-c$, where $c>2\left\lfloor\frac{h}{N}\right\rfloor-1$, so that $s=0, 1, \frac{N-2h+1}{N}$ and $s=-1, -3, -5,\cdots, -\left(2\left\lfloor\frac{h}{N}\right\rfloor-1\right)$ all turn out to be the poles of the integrand on the right side of \eqref{integ}. However, we need to shift the line a little further to the left, that is, to Re$(s)=-c$, where $c>\frac{2h}{N}-1$, the reason for which is two-fold. Firstly, even though $h>N/2$, we can still have $\left\lfloor\frac{h}{N}\right\rfloor=0$. So if the line of integration were shifted such that $c>2\left\lfloor\frac{h}{N}\right\rfloor-1$, this would allow $c$ to take negative values also. But this is undesirable as we would like to capture the pole of the integrand at $s=0$. The second reason for taking $c>\frac{2h}{N}-1$ will be clear soon.

Take the contour $\mathcal{C}$ determined by the line segments $[c_0 - i T, c_0 + i T], [c_0 + i T, -c + i T], [-c + i T, -c - i T]$ and $[-c - i T, c_0 - i T]$. By Cauchy's residue theorem, 
\begin{align}\label{crthm}
&\frac{1}{2\pi i}\left[\int_{c_0-iT}^{c_0+iT}+\int_{c_0+iT}^{-c+iT}+\int_{-c+iT}^{-c-iT}+\int_{-c-iT}^{c_0-iT}\right]\G(s)\zeta(s)\zeta\left(Ns-\left(N-2h\right)\right)x^{-s}\, {\rm d}s \nonumber\\
&\quad=R_{0}+R_{1}+R_{\frac{N-2h+1}{N}}+\sum_{j=1}^{\left\lfloor\frac{h}{N}\right\rfloor}R_{-(2j-1)},
\end{align}
where $R_{a}$ denotes the residue of the integrand at the pole $s=a$. 

As $T\to\infty$, the integrals along the horizontal segments $[c_0 + i T, -c + i T], [-c - i T, c_0 - i T]$ approach zero which is readily seen using Stirling's formula for $\G(s)$ and elementary bounds on the Riemann zeta function. So let $T\to\infty$ in \eqref{crthm} and use \eqref{integ} to deduce that
\begin{align*}
\sum_{n=1}^{\infty}\frac{n^{N-2h}}{e^{n^{N}x}-1}&=J(x)+R_{0}+R_{1}+R_{\frac{N-2h+1}{N}}+\sum_{j=1}^{\left\lfloor h/N\right\rfloor}R_{-(2j-1)},
\end{align*}
where
\begin{equation*}
J(x):=\int_{(-c)}\G(s)\zeta(s)\zeta\left(Ns-\left(N-2h\right)\right)x^{-s}\, {\rm d}s.
\end{equation*}
As shown in \cite{ktyhr},
\begin{align*}
R_{0}+R_{1}+R_{\frac{N-2h+1}{N}}=P(x),
\end{align*}
where $P(x)$ is defined in \eqref{pxfir}. Thus, the sum of residues at all of the poles is
\begin{align*}
P_1(x):=P_1(x;N,h):=
P(x)+\sum_{j=1}^{\left\lfloor\frac{h}{N}\right\rfloor}R_{-(2j-1)}.
\end{align*}
Now
\begin{align*}
R_{-(2j-1)}&=\lim_{s\to-(2j-1)}(s+2j-1)\G(s)\zeta(s)\zeta(Ns-(N-2h))x^{-s}\nonumber\\
&=\frac{-1}{(2j-1)!}\zeta(-(2j-1))\zeta(2h-2jN)x^{2j-1}\nonumber\\
&=\frac{(-1)^{h-jN+1}(2\pi)^{2h-2jN}B_{2j}B_{2h-2jN}}{2(2j)!(2h-2jN)!}x^{2j-1}.
\end{align*}
In the penultimate step above, we used the fact \cite[p.~179, Equation \text{(7.10)}]{con} that $\lim_{s\to -n}(s+n)\G(s)=\frac{(-1)^n}{n!}$, and in the ultimate step we used the fact \cite[p.~266, Equation (20)]{apostol-1998a} $\zeta(-n)=-\frac{B_{n+1}}{n+1}$ along with \eqref{zetaevenint}. 

This verifies the expression for $P_1(x)$ as given in \eqref{p1x}. The only other thing to be done is to show that $J(x)=S(x)$ for $c>\frac{2h}{N}-1$, where $S(x)$ is defined in \eqref{sodd} and \eqref{seven}.

Much of the remainder of the proof is exactly the same as in \cite{ktyhr}, and relies on making the change of variable $s\to 1-s$ in the integral $J(x)$, using the functional equation satisfied by the resulting integrand to simplify it. To avoid duplication of this part of the proof, we refer the reader to \cite[p.~14]{ktyhr}. Thus for Re$(s_1)=c_1=N(1+c)-2h+1$, we obtain
\begin{align*}
J(x)&=(-1)^{h+1}\left(\frac{2\pi}{x}\right)^{\frac{N-2h+1}{N}}\nonumber\\
&\quad\times\frac{1}{2\pi i}\int_{(c_1)}\left(\frac{(2\pi)^{N+1}}{x}\right)^{-\frac{s_1}{N}}\G(s_1)\zeta(s_1)\zeta\left(\frac{s_1+2h-1}{N}\right)C_{N}\left(\frac{\pi(s_1+2h-1)}{2N}\right)\frac{{\rm d}s_1}{N},
\end{align*}
where $C_{N}(z)$ is defined by
\begin{equation*}
C_{N}(z)=\frac{\sin(Nz)}{\sin(z)}.
\end{equation*}
Now we would like to, as in \cite{ktyhr}, expand $\zeta(s_1)\zeta\left(\frac{s_1+2h-1}{N}\right)$ as a Dirichlet series. However, this does not follow the argument in \cite{ktyhr} for, in the proof in there, the authors considered $0<h\leq N/2$, where as in our case, we have $h>N/2$. This is where it helps to have the line of integration shifted in the earlier part of the proof from Re$(s)=c_0$ to Re$(s)=-c$, where $c>\frac{2h}{N}-1$. Hence Re$(s_1)=c_1>1$ and also Re$\left(\frac{s_1+2h-1}{N}\right)=1+c>1$. Thus we can expand $\zeta(s_1)\zeta\left(\frac{s_1+2h-1}{N}\right)$ as
\begin{equation*}
\zeta(s_1)\zeta\left(\frac{s_1+2h-1}{N}\right)=\sum_{m,n=1}^{\infty}n^{-\frac{(2h-1)}{N}}\left(mn^{\frac{1}{N}}\right)^{-s_1}.
\end{equation*} 
The remainder now follows exactly as in \cite{ktyhr} and hence avoiding the repetition, we see that $J(x)$ indeed is equal to $S(x)$ defined in \eqref{sodd} and \eqref{seven}. This completes the proof.

\section{Proof of the generalization of Ramanujan's formula for $\zeta(2m+1)$}\label{oddd}
We begin with an elementary lemma which will be used several times in the sequel.
\begin{lemma}\label{ktywigl}
For $a, u, v\in\mathbb{R}$, we have
\begin{align*}
\frac{\cos(a\sin(u)+uv)-e^{-a\cos(u)}\cos(uv)}{\cosh(a\cos(u))-\cos(a\sin(u))}=2\textup{Re}\left(\frac{e^{iuv}}{\exp{\left(ae^{-iu}\right)}-1}\right).
\end{align*}
\end{lemma}
\begin{proof}
Note that the right hand side can be simplified to
\begin{align*}
2\textup{Re}\left(\frac{e^{iuv}}{\exp{\left(ae^{-iu}\right)}-1}\right)=2\textup{Re}\left(\frac{\cos(uv)+i\sin(uv)}{e^{a\cos(u)}\left(\cos(a\sin(u))-i\sin(a\sin(u))\right)-1}\right).
\end{align*}
Multiply the numerator and the denominator by the conjugate of the denominator so that the right side becomes
{\allowdisplaybreaks\begin{align*}
&2\textup{Re}\left(\frac{\left(\cos(uv)+i\sin(uv)\right)\left(e^{a\cos(u)}\cos(a\sin(u))-1+ie^{a\cos(u)}\sin(a\sin(u))\right)}{e^{2a\cos(u)}-2e^{a\cos(u)}\cos(a\sin(u))+1}\right)\nonumber\\
&=\frac{2\left(e^{a\cos(u)}\cos(a\sin(u)+uv)-\cos(uv)\right)}{e^{2a\cos(u)}+1-2e^{a\cos(u)}\cos(a\sin(u))}\nonumber\\
&=\frac{\cos(a\sin(u)+uv)-e^{-a\cos(u)}\cos(uv)}{\cosh(a\cos(u))-\cos(a\sin(u))}.
\end{align*}}
\end{proof}

\begin{proof}[Theorem \textup{\ref{zetagen}}][]
We first show that all of the generalized Lambert series occurring in Theorem \ref{zetagen} converge. This is easily seen to be true for $\sum_{n=1}^{\infty}\frac{n^{-2Nm-1}}{\textup{exp}\left((2n)^{N}\a\right)-1}$ since $\a>0$ and $N>0$. For the remaining ones, it suffices to show that Re$\left((2n)^{\frac{1}{N}}\b e^{\frac{i\pi j}{N}}\right)>0$, that is, Re$\left(e^{\frac{i\pi j}{N}}\right)>0$ for $-\frac{N-1}{2}\leq j\leq\frac{N-1}{2}$. This is obvious since
\begin{equation*}
-\frac{\pi}{2}<-\frac{\pi(N-1)}{2N}\leq\frac{\pi j}{N}\leq\frac{\pi(N-1)}{2N}<\frac{\pi}{2}
\end{equation*}
implies that $\cos\left(\frac{\pi j}{N}\right)>0$ for $-\frac{N-1}{2}\leq j\leq\frac{N-1}{2}$.

Let $N$ be an odd positive integer. In Theorem \ref{ktycomp}, let $h=\frac{N+1}{2}+Nm$, where $m\in\mathbb{Z}\backslash\{0\}$. Upon simplification, this gives
{\allowdisplaybreaks\begin{align}\label{oddzetax}
&\frac{1}{2}\zeta(2Nm+1)+\sum_{n=1}^{\infty}\frac{n^{-2Nm-1}}{\textup{exp}\left(n^{N}x\right)-1}\nonumber\\
&=\frac{(-1)^{m}}{2N}\left(\frac{x}{2\pi}\right)^{2m}\zeta(2m+1)+\frac{1}{x}\zeta(N+1+2Nm)\nonumber\\
&\quad+\frac{1}{2}(-1)^{m+\frac{N+3}{2}}(2\pi)^{N+1+2Nm}\sum_{j=1}^{\left\lfloor\frac{N+1}{2N}+m\right\rfloor}\frac{(-1)^jB_{2j}B_{N+1+2N(m-j)}}{(2\pi)^{2jN}(2j)!(N+1+2N(m-j))!}x^{2j-1}\nonumber\\
&\quad+\frac{(-1)^{m+\frac{N+3}{2}}}{N}\left(\frac{x}{2\pi}\right)^{2m}\sum_{n=1}^{\infty}\frac{1}{n^{2m+1}}\left\{\frac{1}{e^{a}-1}+\sum_{j=1}^{\frac{N-1}{2}}\frac{\cos(a\sin(u)+uv)-e^{-a\cos(u)}\cos(uv)}{\cosh(a\cos(u))-\cos(a\sin(u))}\right\},
\end{align}}
where $a=2A_N\left(\frac{n}{x}\right), u=\frac{\pi j}{N}$ and $v=2h-1=(2m+1)N$. 

Let $x=2^{N}\a$ and let $\a\b^{N}=\pi^{N+1}$ so that $\b=2\pi\left(\frac{\pi}{x}\right)^{1/N}$, and hence $a=(2n)^{\frac{1}{N}}\b$. 

We write
\begin{align}\label{powerpiab1}
\left(\frac{x}{2\pi}\right)^{2m}=\left(\frac{2^{N}\a}{2\a^{\frac{1}{N+1}}\b^{\frac{N}{N+1}}}\right)^{2m}=2^{2m(N-1)}\a^{\frac{2Nm}{N+1}}\b^{-\frac{2Nm}{N+1}},
\end{align}
and
\begin{align}\label{powerpiab2}
(2\pi)^{N+1+2Nm-2jN}x^{2j-1}&=(2\pi)^{N+1+2Nm-2jN}\left(2^{N}\a\right)^{2j-1}\nonumber\\
&=2^{2Nm+1}\a^{2j-1}\left(\left(\frac{\pi^{N+1}}{\a}\right)^{1+\frac{2N}{N+1}(m-j)}\a^{1+\frac{2N}{N+1}(m-j)}\right)\nonumber\\
&=2^{2Nm+1}\a^{2j+\frac{2N}{N+1}(m-j)}\b^{N+\frac{2N^2}{N+1}(m-j)}.
\end{align}
Now substitute \eqref{powerpiab1} in the first and the last expressions on the right side of \eqref{oddzetax}, and \eqref{powerpiab2} in the third expression. (Since $N+1+2Nm$ is even, using \eqref{zetaevenint}, the second expression $\zeta(N+1+2Nm)/x$ in \eqref{oddzetax} can be absorbed into the third one as its $j=0$ term.) Then divide both sides of the resulting equation by $\a^{\frac{2Nm}{N+1}}$ and invoke Lemma \ref{ktywigl} for simplifying the last expression on the right of \eqref{oddzetax} to arrive at
\begin{align}\label{ramng}
&\a^{-\frac{2Nm}{N+1}}\left(\frac{1}{2}\zeta(2Nm+1)+\sum_{n=1}^{\infty}\frac{n^{-2Nm-1}}{\textup{exp}\left((2n)^{N}\a\right)-1}\right)\nonumber\\
&=\left(-\b^{\frac{2N}{N+1}}\right)^{-m}\frac{2^{2m(N-1)}}{N}\Bigg(\frac{1}{2}\zeta(2m+1)+(-1)^{\frac{N+3}{2}}\sum_{n=1}^{\infty}\frac{1}{n^{2m+1}}\Bigg\{\frac{1}{\exp{\left((2n)^{\frac{1}{N}}\b\right)}-1}\nonumber\\
&\qquad\qquad\qquad\qquad\qquad\qquad+\sum_{j=1}^{\frac{N-1}{2}}2\textup{Re}\Bigg(\frac{e^{i\pi j(2m+1)}}{\exp{\left((2n)^{\frac{1}{N}}\b e^{\frac{-i\pi j}{N}}\right)}-1}\Bigg)\Bigg\}\Bigg)\nonumber\\
&\quad+(-1)^{m+\frac{N+3}{2}}2^{2Nm}\sum_{j=0}^{\left\lfloor\frac{N+1}{2N}+m\right\rfloor}\frac{(-1)^jB_{2j}B_{N+1+2N(m-j)}}{(2j)!(N+1+2N(m-j))!}\a^{\frac{2j}{N+1}}\b^{N+\frac{2N^2(m-j)}{N+1}}.
\end{align}
Note that $e^{i\pi j(2m+1)}=(-1)^{j}$, thus 
{\allowdisplaybreaks\begin{align}\label{gad}
&\sum_{n=1}^{\infty}\frac{1}{n^{2m+1}}\Bigg\{\frac{1}{\exp{\left((2n)^{\frac{1}{N}}\b\right)}-1}+\sum_{j=1}^{\frac{N-1}{2}}2\textup{Re}\Bigg(\frac{e^{i\pi j(2m+1)}}{\exp{\left((2n)^{\frac{1}{N}}\b e^{\frac{-i\pi j}{N}}\right)}-1}\Bigg)\Bigg\}\nonumber\\
&=\sum_{n=1}^{\infty}\frac{n^{-2m-1}}{\exp{\left((2n)^{\frac{1}{N}}\b\right)}-1}\nonumber\\
&\quad+\sum_{j=1}^{\frac{N-1}{2}}(-1)^{j}\left(\sum_{n=1}^{\infty}\frac{n^{-2m-1}}{\exp{\left((2n)^{\frac{1}{N}}\b e^{\frac{i\pi j}{N}}\right)}-1}+\sum_{n=1}^{\infty}\frac{n^{-2m-1}}{\exp{\left((2n)^{\frac{1}{N}}\b e^{\frac{-i\pi j}{N}}\right)}-1}\right)\nonumber\\
&=\sum_{j=\frac{-(N-1)}{2}}^{\frac{N-1}{2}}(-1)^{j}\sum_{n=1}^{\infty}\frac{n^{-2m-1}}{\textup{exp}\left((2n)^{\frac{1}{N}}\b e^{\frac{i\pi j}{N}}\right)-1}.
\end{align}}
Substituting \eqref{gad} in \eqref{ramng} leads to \eqref{zetageneqn}, thereby completing the proof of Theorem \ref{zetagen}.
\end{proof}
The following corollary of Theorem \ref{zetagen} gives a nice relation between $\zeta(3)$ and $\zeta(7)$.
\begin{corollary}\label{corz3z7}
Let $\omega:=\frac{-1+\sqrt{3}i}{2}$ denote a cube-root of unity. Let $\a,\b>0$ such that $\a\b^3=\pi^4$. Then
\begin{align}\label{z3z7}
&\a^{-\frac{3}{2}}\left(\frac{1}{2}\zeta(7)+\sum_{n=1}^{\infty}\frac{1}{n^7(\exp{\left(8n^3\a\right)}-1)}\right)\nonumber\\
&=\frac{\b^{\frac{15}{2}}}{748440}+\frac{\a^{\frac{1}{2}}\b^3}{135}-\frac{16}{3}\b^{-\frac{3}{2}}\Bigg(\frac{1}{2}\zeta(3)-\sum_{n=1}^{\infty}\frac{n^{-3}}{\textup{exp}\left((2n)^{\frac{1}{3}}\b\right)-1}\nonumber\\
&\quad+\sum_{n=1}^{\infty}\frac{n^{-3}}{\exp{\big(-(2n)^{\frac{1}{3}}\b\omega\big)-1}}+\sum_{n=1}^{\infty}\frac{n^{-3}}{\exp{\big(-(2n)^{\frac{1}{3}}\b\omega^2\big)-1}}\Bigg).
\end{align}
\end{corollary}     
\begin{proof}
Let $N=3$ and $m=1$ in Theorem \ref{zetagen}. This gives
{\allowdisplaybreaks\begin{align*}
&\a^{-\frac{3}{2}}\left(\frac{1}{2}\zeta(7)+\sum_{n=1}^{\infty}\frac{1}{n^7(\exp{\left(8n^3\a\right)}-1)}\right)\nonumber\\
&=-\frac{16}{3}\b^{-\frac{3}{2}}\left(\frac{1}{2}\zeta(3)-\sum_{j=-1}^{1}(-1)^{j}\sum_{n=1}^{\infty}\frac{n^{-3}}{\textup{exp}\left((2n)^{\frac{1}{3}}\b e^{\frac{i\pi j}{3}}\right)-1}\right)\nonumber\\
&\quad+2^6\sum_{j=0}^{1}\frac{(-1)^jB_{2j}B_{10-6j}}{(2j)!(10-6j)!}\a^{\frac{j}{2}}\b^{3+\frac{9(1-j)}{2}}.
\end{align*}}
This gives \eqref{z3z7} upon simplification.
\end{proof}

\subsection{A generalization of the transformation for $\log\eta(z)$}\label{eta}
\begin{proof}[Theorem \textup{\ref{zetagenm0}}][]
Since the proof is essentially similar to that of Theorem \ref{zetagen}, we only indicate the places where it differs from that of the latter. To that end, we let $h=\frac{N+1}{2}$ in Theorem \ref{ktycomp}. This gives for Re$(s)=c_0>1$,
\begin{align}\label{integ1}
\sum_{n=1}^{\infty}\frac{1}{n\left(e^{n^{N}x}-1\right)}=\frac{1}{2\pi i}\int_{(c_{0})}\G(s)\zeta(s)\zeta\left(Ns+1\right)x^{-s}\, {\rm d}s.
\end{align}
We then shift the line of integration from Re$(s)=c_0$ to Re$(s)=-c$, where $c>1/N$. In the shifting process, we encounter a simple pole of the integrand at $s=1$ with the residue $R_{1}=\zeta(N+1)x^{-1}$. The pole at $s=0$ is of order two, as both $\G(s)$ and $\zeta(Ns+1)$ have simple poles at $s=0$. The residue $R_{0}$ at this simple pole can be calculated to be
\begin{equation*}
R_0=\frac{1}{2N}\left(\g (1-N)-\log(2\pi)+\log(x)\right).
\end{equation*}
The only other pole of the integrand that we encounter is at $s=-1$, and that too only when $N=1$. Its residue is $R_{-1}=-x/24$. Now proceeding exactly along the similar lines of the proof of Theorem \ref{zetagen}, one obtains
\begin{align}\label{befgxn}
\sum_{n=1}^{\infty}\frac{1}{n\left(e^{n^{N}x}-1\right)}&=-\frac{1}{2N}\left((N-1)\g+\log(2\pi)-\log(x)\right)+\frac{\zeta(N+1)}{x}+g(x, N)\nonumber\\
&\quad+\frac{1}{N}(-1)^{\frac{N+3}{2}}\sum_{j=-\frac{(N-1)}{2}}^{\frac{N-1}{2}}(-1)^j\sum_{n=1}^{\infty}\frac{1}{n \left(\exp{\left(2\pi\left(\frac{2\pi n}{x}\right)^{\frac{1}{N}}e^{\frac{i\pi j}{N}}\right)}-1\right)},
\end{align}
where
\begin{equation}
g(x, N):=\begin{cases}
-\frac{x}{24},\hspace{3mm}\text{if}\hspace{1mm} N=1,\\
0,\hspace{3mm}\text{otherwise}.
\end{cases}
\end{equation}
Note that 
\begin{equation*}
g(x, N)=(-1)^{\frac{N+3}{2}}2^N\pi^{N+1}\sum_{j=1}^{\left\lfloor\frac{N+1}{2N}\right\rfloor}\left(\frac{-1}{4\pi^2}\right)^{jN}\frac{B_{2j}B_{N+1-2Nj}}{(2j)!(N+1-2Nj)!}x^{2j-1}.
\end{equation*}
Using the above form of $g(x, N)$ in \eqref{befgxn}, letting $x=2^{N}\a$, $\a\b^{N}=\pi^{N+1}$ and using \eqref{powerpiab2} with $m=0$, we arrive at \eqref{zetagenm0eqn}.
\end{proof}
The case $N=3$ of Theorem \ref{zetagenm0} gives the following new result:
\begin{corollary}\label{abpi}
Let $\omega:=\frac{-1+\sqrt{3}i}{2}$ denote a cube-root of unity. For $\a,\b>0$ such that $\a\b^3=\pi^4$, we have
\begin{align*}
&\sum_{n=1}^{\infty}\frac{1}{n(\exp{\left(8n^3\a\right)}-1)}+\frac{1}{3}\Bigg(\sum_{n=1}^{\infty}\frac{1}{n\left(\exp{((2n)^{\frac{1}{3}}\b)}-1\right)}\nonumber\\
&\quad-\sum_{n=1}^{\infty}\frac{1}{n\left(\textup{exp}\left(-(2n)^{\frac{1}{3}}\b\omega\right)-1\right)}-\sum_{n=1}^{\infty}\frac{1}{n\left(\textup{exp}\left(-(2n)^{\frac{1}{3}}\b\omega^{2}\right)-1\right)}\Bigg)\nonumber\\
&=\frac{1}{3}\left(\log 2-\gamma\right)+\frac{1}{8}\log\left(\frac{\a}{\b}\right)+\frac{\b^3}{720}.
\end{align*}
\end{corollary}
\section{Proof of the generalization of Wigert's formula for $\zeta\left(\frac{1}{N}\right)$}\label{evenn}
In this section we prove the counterpart of Theorem \ref{zetagen} for $N$ even. 
\begin{proof}[Theorem \textup{\ref{neven}}][]
Note that the series $\displaystyle\sum_{n=1}^{\infty}\frac{n^{-2Nm}}{\textup{exp}\left((2n)^{N}\a\right)-1}$ converges since $\a>0$ and $N>0$. It suffices to show for the remaining Lambert series that \newline Re$\left(\exp{\left(\frac{i\pi(2j+1)}{2N}\right)}\right)>0$ for $0\leq j\leq\frac{N}{2}-1$. To that end, note that
\begin{equation*}
0<\frac{\pi}{2N}\leq\frac{\pi(2j+1)}{2N}\leq\frac{\pi(N-1)}{2N}<\frac{\pi}{2},
\end{equation*}
so that $\cos\left(\frac{\pi(2j+1)}{2N}\right)>0$ and hence the series converge.

Let $N$ be an even positive integer. In Theorem \ref{ktycomp}, let $h=\frac{N}{2}+Nm$, where $m\in\mathbb{Z}$. After some simplification, this gives
{\allowdisplaybreaks\begin{align}\label{genwig1s}
&\frac{1}{2}\zeta(2Nm)+\sum_{n=1}^{\infty}\frac{n^{-2Nm}}{e^{n^{N}x}-1}\nonumber\\
&=\frac{1}{x}\zeta((2m+1)N)+\frac{1}{N}\G\left(\frac{1-2Nm}{N}\right)\zeta\left(\frac{1-2Nm}{N}\right)x^{-\frac{(1-2Nm)}{N}}\nonumber\\
&\quad+\frac{(-1)^{\frac{N}{2}+1}}{N}\left(\frac{2\pi}{x}\right)^{\frac{1-2Nm}{N}}\sum_{n=1}^{\infty}\frac{1}{n^{2m+1-\frac{1}{N}}}\sum_{j=0}^{\frac{N}{2}-1}\frac{\cos(a\sin(u)+uv)-e^{-a\cos(u)}\cos(uv)}{\cosh(a\cos(u))-\cos(a\sin(u))}\nonumber\\
&\quad+(-1)^{\frac{N}{2}+1}2^{(2m+1)N-1}\pi^{(2m+1)N}\sum_{j=1}^{m}\frac{B_{2j}B_{(2m+1-2j)N}}{(2j)!((2m+1-2j)N)!}(2\pi)^{-2jN}x^{2j-1},
\end{align}}
where $a=2A_N\left(\frac{n}{x}\right), u=\frac{\pi (2j+1)}{2N}$ and $v=2h-1=(2m+1)N-1$.

Note that $(2m+1)N$ is even. Hence applying \eqref{zetaevenint}, it is easy to see that the expression $\zeta((2m+1)N)/x$ can be absorbed into the last expression in \eqref{genwig1s} as its $j=0$ term. 

Let $x=2^{N}\a$ and let $\a\b^{N}=\pi^{N+1}$ so that $\b=2\pi\left(\frac{\pi}{x}\right)^{1/N}$, and hence $a=(2n)^{\frac{1}{N}}\b$.

\noindent
Consider the second expression on the right side of \eqref{genwig1s}. Since $x=2^{N}\a=\pi\left(\frac{2\pi}{\b}\right)^{N}$, we get
\begin{equation*}
x^{-\frac{(1-2Nm)}{N}}=\pi^{-\frac{(1-2Nm)}{N}}\left(\frac{2\pi}{\b}\right)^{2Nm-1};
\end{equation*}
use the functional equation \eqref{zetafe} in the form $\pi^{-s}\G(s)\zeta(s)=2^{s-1}\zeta(1-s)/\cos\left(\frac{\pi s}{2}\right)$ to get
\begin{align}\label{e1c}
&\frac{1}{N}\G\left(\frac{1-2Nm}{N}\right)\zeta\left(\frac{1-2Nm}{N}\right)x^{-\frac{(1-2Nm)}{N}}\nonumber\\
&=\frac{(-1)^{m}}{N}\frac{2^{\frac{1}{N}-(2m+1)}}{\cos\left(\frac{\pi}{2N}\right)}\zeta\left(2m+1-\frac{1}{N}\right)\left(\frac{2\pi}{\b}\right)^{2Nm-1}.
\end{align}
From \eqref{powerpiab2}, 
\begin{align}\label{e2c}
(2\pi)^{(2m+1)N-1-2jN}x^{2j-1}=\frac{2^{2Nm-1}}{\pi^2}\a^{2j+\frac{2N}{N+1}(m-j)}\b^{N+\frac{2N^2}{N+1}(m-j)},
\end{align}
We now want to write the third expression in \eqref{genwig1s} in terms of generalized Lambert series. Firstly using \eqref{powerpiab1}, we see that
\begin{equation}\label{fly}
\left(\frac{2\pi}{x}\right)^{\frac{1-2Nm}{N}}=2^{(N-1)\left(2m-\frac{1}{N}\right)}\a^{\frac{2Nm-1}{N+1}}\b^{\frac{1-2Nm}{N+1}}.
\end{equation}
Secondly, invoking Lemma \ref{ktywigl} and using the fact that $\exp{\left(\tfrac{1}{2}\left(i\pi(2j+1)(2m+1)\right)\right)}=i(-1)^{j+m}$ in the second step below, we deduce that
{\allowdisplaybreaks\begin{align}\label{sly}
&\sum_{j=0}^{\frac{N}{2}-1}\frac{\cos(a\sin(u)+uv)-e^{-a\cos(u)}\cos(uv)}{\cosh(a\cos(u))-\cos(a\sin(u))}\nonumber\\&=2\sum_{j=0}^{\frac{N}{2}-1}\textup{Re}\left(\frac{\exp{\left(\frac{i\pi}{2N}(2j+1)((2m+1)N-1)\right)}}{\exp{\left((2n)^{\frac{1}{N}}\b e^{-\frac{i\pi(2j+1)}{2N}}\right)}}\right)\nonumber\\
&=2(-1)^{j+m+1}\sum_{j=0}^{\frac{N}{2}-1}\textup{Im}\left(\frac{\exp{\left(-\frac{i\pi(2j+1)}{2N}\right)}}{\exp{\left((2n)^{\frac{1}{N}}\b e^{-\frac{i\pi(2j+1)}{2N}}\right)}}\right)\nonumber\\
&=i(-1)^{j+m+1}\sum_{j=0}^{\frac{N}{2}-1}\left(\frac{\exp{\left(\frac{i\pi(2j+1)}{2N}\right)}}{\exp{\left((2n)^{\frac{1}{N}}\b e^{\frac{i\pi(2j+1)}{2N}}\right)}}-\frac{\exp{\left(-\frac{i\pi(2j+1)}{2N}\right)}}{\exp{\left((2n)^{\frac{1}{N}}\b e^{-\frac{i\pi(2j+1)}{2N}}\right)}}\right)
\end{align}}
Thus from \eqref{fly} and \eqref{sly}, the third expression in \eqref{genwig1s} can be written as
\begin{align}\label{fsly}
&\frac{(-1)^{\frac{N}{2}+1}}{N}\left(\frac{2\pi}{x}\right)^{\frac{1-2Nm}{N}}\sum_{n=1}^{\infty}\frac{1}{n^{2m+1-\frac{1}{N}}}\sum_{j=0}^{\frac{N}{2}-1}\frac{\cos(a\sin(u)+uv)-e^{-a\cos(u)}\cos(uv)}{\cosh(a\cos(u))-\cos(a\sin(u))}\nonumber\\
&=i\frac{(-1)^{\frac{N}{2}+m}}{N}2^{(N-1)\left(2m-\frac{1}{N}\right)}\a^{\frac{2Nm-1}{N+1}}\b^{\frac{1-2Nm}{N+1}}\sum_{j=0}^{\frac{N}{2}-1}(-1)^j\nonumber\\
&\quad\times\Bigg(e^{\frac{i\pi(2j+1)}{2N}}\sum_{n=1}^{\infty}\frac{n^{\frac{1}{N}-(2m+1)}}{\exp{\left((2n)^{\frac{1}{N}}\b e^{\frac{i\pi(2j+1)}{2N}}\right)}-1}-e^{\frac{-i\pi(2j+1)}{2N}}\sum_{n=1}^{\infty}\frac{n^{\frac{1}{N}-(2m+1)}}{\exp{\left((2n)^{\frac{1}{N}}\b e^{-\frac{i\pi(2j+1)}{2N}}\right)}-1}\Bigg).
\end{align}
Now substitute $2^{N}\a$ for $x$ on the left side of \eqref{genwig1s} and rewrite its right side as follows. Absorb $\zeta((2m+1)N)/x$ as the $j=0$ term of the last expression, substitute \eqref{e1c} and \eqref{e2c} in the second and the last expressions respectively, and substitute \eqref{fsly} for the third expression. Upon doing this, divide both sides of the resulting equivalent of \eqref{genwig1s} by $\a^{\frac{2Nm-1}{N+1}}$. This gives \eqref{neveneqn} after simplification and thus completes the proof.
\end{proof}

\section{Applications}\label{app}
This section is devoted to applications of our Theorems \ref{zetagen}, \ref{zetagenm0} and \ref{neven} towards proving results on transcendence of certain expressions. It is important to note that these transcendence results might have been exceedingly difficult to prove if we did not have our transformations to begin with. We start with a Zudilin-type result.
\begin{corollary}\label{transc}
Let $N$ and $m$ be any two fixed odd natural numbers. Then at least one of the expressions
\begin{equation*}
\zeta(2m+1), \zeta(2Nm+1), \sum_{n=1}^{\infty}\frac{n^{-2Nm-1}}{\textup{exp}\left((2n)^N\pi\right)-1},\hspace{2mm}and\hspace{2mm}\sum_{n=1}^{\infty}\frac{1}{n^{2m+1}}\textup{Re}\bigg(\frac{1}{\textup{exp}\big((2n)^{\frac{1}{N}}\pi e^{\frac{i\pi j}{N}}\big)-1}\bigg), 
\end{equation*}
where $j$ takes every value from $0$ to $\frac{N-1}{2}$,
is transcendental. The above conclusion also holds for any fixed even natural number $m$ so long as $N$ is any fixed odd number strictly greater than $1$.
\end{corollary}
\begin{proof}
We prove only the case when $m$ is an odd natural number. Let $\a=\b=\pi$ in Theorem \ref{zetagen} and multiply both sides by $2\pi^{\frac{2Nm}{N+1}}$. For $m$ odd, this gives 
{\allowdisplaybreaks\begin{align}\label{lerchgen}
&\zeta(2Nm+1)+\frac{2^{2m(N-1)}}{N}\zeta(2m+1)+2\sum_{n=1}^{\infty}\frac{n^{-2Nm-1}}{\exp{\left((2n)^N\pi\right)}-1}\nonumber\\
&+\frac{1}{N}(-1)^{\frac{N+3}{2}}2^{2m(N-1)+1}\nonumber\\
&\quad\times\left(\sum_{n=1}^{\infty}\frac{n^{-2m-1}}{\exp{\left((2n)^{\frac{1}{N}}\pi\right)}-1}+2\sum_{j=1}^{\frac{N-1}{2}}\frac{1}{n^{2m+1}}\textup{Re}\left(\frac{1}{\textup{exp}\left((2n)^{\frac{1}{N}}\pi e^{\frac{i\pi j}{N}}\right)-1}\right)\right)\nonumber\\
&=(-1)^{\frac{N+5}{2}}2^{2Nm+1}\pi^{N(2m+1)}\sum_{j=0}^{\left\lfloor\frac{N+1}{2N}+m\right\rfloor}\frac{(-1)^jB_{2j}B_{N+1+2N(m-j)}}{(2j)!(N+1+2N(m-j))!}\pi^{2j(1-N)}
\end{align}}
upon simplification. This, by the way, is a generalization of a Lerch's result \eqref{lerch} for $N$ odd such that $N\geq 1$. For $N=1$ and $m$ odd, the result stated in the corollary can be obtained from \eqref{lerch} since the resulting right-hand side is a non-zero rational multiple of $\pi^{2m+1}$ so that either $\zeta(2m+1)$ or $\sum_{n=1}^{\infty}\displaystyle\frac{n^{-2m-1}}{e^{2\pi n}-1}$ in \eqref{lerch} must be transcendental as $\pi$ is known to be transcendental. Now for $N>1$, note that the right-hand side of \eqref{lerchgen} is nothing but a polynomial in $\pi$ with rational coefficients. The conclusion then follows again from the transcendence of $\pi$.

The case when $m$ is an even natural number can be proved in a very similar way and hence a proof is omitted. It is nice to see in this case, however, that while substituting $N=1$ after letting $\a=\b=\pi$ in Theorem \ref{zetagen} and multiplying both sides by $\pi^{\frac{2Nm}{N+1}}$ leads to
\begin{equation*}
\sum_{j=0}^{m+1}\frac{(-1)^jB_{2j}B_{2m+2-2j}}{(2j)!(2m+2-2j)!}=0,
\end{equation*} 
and hence no information on the arithmetic nature of $\zeta(2m+1)$, $m$ even, we \emph{do} get a transcendence result for $N>1$ precisely because the expression involving Bernoulli numbers again turns out to be a polynomial in $\pi$ with rational coefficients. Thus, the case $N>1$ gives information on the arithmetic nature of $\zeta(2m+1)$ for $m$ even.
\end{proof}
The above corollary, in turn, readily gives the following Rivoal-type results. Hence proofs are omitted.
\begin{corollary}\label{rtype}
Let $N$ be any fixed odd positive integer. The set
\begin{align*}
&\bigcup_{k=0}^{\infty}\bigg\{\zeta(4k+3),\hspace{2mm} \sum_{n=1}^{\infty}\frac{n^{-2N(2k+1)-1}}{\textup{exp}\left((2n)^N\pi\right)-1},\hspace{2mm} \sum_{n=1}^{\infty}\frac{1}{n^{4k+3}}\textup{Re}\bigg(\frac{1}{\textup{exp}\big((2n)^{\frac{1}{N}}\pi e^{\frac{i\pi j}{N}}\big)-1}\bigg)\nonumber\\
&\qquad: 0\leq j\leq \frac{N-1}{2}\bigg\},
\end{align*}
contains infinitely many transcendental numbers. Also, for any fixed odd positive integer $N>1$, the set 
\begin{align*}
&\bigcup_{k=1}^{\infty}\bigg\{\zeta(4k+1),\hspace{2mm} \sum_{n=1}^{\infty}\frac{n^{-4Nk-1}}{\textup{exp}\left((2n)^N\pi\right)-1}, \hspace{2mm}\sum_{n=1}^{\infty}\frac{1}{n^{4k+1}}\textup{Re}\bigg(\frac{1}{\textup{exp}\big((2n)^{\frac{1}{N}}\pi e^{\frac{i\pi j}{N}}\big)-1}\bigg)\nonumber\\
&\quad: 0\leq j\leq \frac{N-1}{2}\bigg\},
\end{align*}
where $0\leq j\leq \frac{N-1}{2}$, contains infinitely many transcendental numbers.
\end{corollary}
Finding if Euler's constant $\gamma$ is algebraic or transcendental is a famous unsolved problem in Mathematics. In fact, it is not even known whether $\gamma$ is irrational. Note that when $N=1$, the expression involving $\gamma$ in Theorem \ref{zetagenm0} vanishes. However, when $N$ is any odd number greater than $1$, Theorem \ref{zetagenm0} gives the following interesting result on transcendence:
\begin{corollary}\label{gammatrans}
For every odd positive integer $N>1$, at least one of
\begin{equation*}
\gamma,\hspace{2mm}\sum_{n=1}^{\infty}\frac{1}{n(\exp{\left(2\pi n^N\right)}-1)},\hspace{2mm}\text{and}\hspace{2mm}\sum_{n=1}^{\infty}\frac{1}{n}\textup{Re}\left(\frac{1}{\textup{exp}\left(2\pi n^{\frac{1}{N}} e^{\frac{i\pi j}{N}}\right)-1}\right),
\end{equation*}
where $j$ takes every value from $0$ to $\frac{N-1}{2}$, is transcendental.
\end{corollary}
\begin{proof}
Let $\a=2^{1-N}\pi$ and $\b=2^{1-\frac{1}{N}}\pi$ in Theorem \ref{zetagenm0}. Upon using Remark \ref{ng1}, this gives 
\begin{align}
&\frac{(N-1)\g}{2N}+\sum_{n=1}^{\infty}\frac{1}{n(\exp{\left(2\pi n^N\right)}-1)}-\frac{(-1)^{\frac{N+3}{2}}}{N}\sum_{j=-\frac{(N-1)}{2}}^{\frac{N-1}{2}}(-1)^j\sum_{n=1}^{\infty}\frac{1}{n}\frac{1}{\left(\textup{exp}\left(2\pi n^{\frac{1}{N}} e^{\frac{i\pi j}{N}}\right)-1\right)}\nonumber\\
&=\frac{2^{N-1}(-1)^{\frac{N+3}{2}}B_{N+1}\,\pi^{N}}{(N+1)!}.
\end{align}
The important thing to note here is that the expressions involving logarithm vanish. Due to this, we can now argue in a similar way as in the proofs of some of the earlier results on transcendence that the right side of the above equation, being a rational multiple of $\pi^{N}$, implies that at least one of 
\begin{equation*}
\gamma,\hspace{2mm}\sum_{n=1}^{\infty}\frac{1}{n(\exp{\left(2\pi n^N\right)}-1)},\hspace{2mm}\text{and}\hspace{2mm}\sum_{n=1}^{\infty}\frac{1}{n}\textup{Re}\bigg(\frac{1}{\textup{exp}\big(2\pi n^{\frac{1}{N}} e^{\frac{i\pi j}{N}}\big)-1}\bigg),
\end{equation*}
where $j$ takes every value from $0$ to $\frac{N-1}{2}$, is transcendental.
\end{proof}
Since the above result holds for \emph{every} odd positive integer $N>1$, it readily implies the following important criterion on transcendence of Euler's constant $\gamma$.
\begin{corollary}\label{tceuler}
If the set
\begin{equation*}
\bigcup_{\ell=1}^{\infty}\bigg\{\sum_{n=1}^{\infty}\frac{1}{n(\exp{\left(2\pi n^{2\ell+1}\right)}-1)}, \hspace{2mm}\sum_{n=1}^{\infty}\frac{1}{n}\textup{Re}\bigg(\frac{1}{\textup{exp}\big(2\pi n^{\frac{1}{2\ell+1}} e^{\frac{i\pi j}{2\ell+1}}\big)-1}\bigg):0\leq j\leq\ell\bigg\},
\end{equation*}
contains only finitely many transcendental numbers, then $\gamma$ must be transcendental.
\end{corollary}
Theorem \ref{neven} also gives the following result on transcendence. 
\begin{corollary}\label{eventrans}
Let $N$ be any even positive integer and $m$ any integer. 
At least one of 
\begin{equation*}
\zeta\left(2m+1-\frac{1}{N}\right), \sum_{n=1}^{\infty}\frac{n^{-2Nm}}{\exp{\left((2n)^{N}\pi\right)}-1},\hspace{2mm}\text{and}\hspace{2mm} \sum_{n=1}^{\infty}\textup{Im}\Bigg(\frac{e^{\frac{i\pi(2j+1)}{2N}}n^{\frac{1}{N}-(2m+1)}}{\exp{\left((2n)^{\frac{1}{N}}\pi e^{\frac{i\pi(2j+1)}{2N}}\right)}-1}\Bigg),
\end{equation*}
where $j$ takes every value from $0$ to $\frac{N}{2}-1$, is transcendental.
\end{corollary}
\begin{proof}
Let $\a=\b=\pi$ in Theorem \ref{neven}. This gives us an analogue of Lerch's formula \eqref{lerch}:
{\allowdisplaybreaks\begin{align*}
&\sum_{n=1}^{\infty}\frac{n^{-2Nm}}{\exp{\left((2n)^{N}\pi\right)}-1}\nonumber\\
&\quad-\frac{(-1)^m}{N}2^{(N-1)\left(2m-\frac{1}{N}\right)}\Bigg(\frac{\zeta\left(2m+1-\frac{1}{N}\right)}{2\cos\left(\frac{\pi}{2N}\right)}\nonumber\\
&\qquad\qquad\qquad+2(-1)^{\frac{N}{2}+1}\sum_{j=0}^{\frac{N}{2}-1}(-1)^j\sum_{n=1}^{\infty}\frac{1}{n^{(2m+1)-\frac{1}{N}}}\textup{Im}\Bigg(\frac{e^{\frac{i\pi(2j+1)}{2N}}}{\exp{\left((2n)^{\frac{1}{N}}\pi e^{\frac{i\pi(2j+1)}{2N}}\right)}-1}\Bigg)\Bigg)\nonumber\\
&=-\frac{1}{2}\zeta(2Nm)+(-1)^{\frac{N}{2}+1}2^{2Nm-1}\pi^{(2m+1)N-1}\sum_{j=0}^{m}\frac{B_{2j}B_{(2m+1-2j)N}}{(2j)!((2m+1-2j)N)!}\pi^{2j(1-N)}.
\end{align*}}
Using \eqref{zetaevenint}, the right side is seen to be a polynomial in $\pi$ with rational coefficients. Thus we see that at least one of 
\begin{equation*}
\frac{\zeta\left(2m+1-\frac{1}{N}\right)}{\cos\left(\frac{\pi}{2N}\right)}, \sum_{n=1}^{\infty}\frac{n^{-2Nm}}{\exp{\left((2n)^{N}\pi\right)}-1},\hspace{2mm}\text{and}\hspace{2mm} \sum_{n=1}^{\infty}\textup{Im}\Bigg(\frac{e^{\frac{i\pi(2j+1)}{2N}}n^{\frac{1}{N}-(2m+1)}}{\exp{\left((2n)^{\frac{1}{N}}\pi e^{\frac{i\pi(2j+1)}{2N}}\right)}-1}\Bigg),
\end{equation*}
where $0\leq j\leq \frac{N}{2}-1$, is transcendental. Along with the fact that $\cos(\pi/4)=1/\sqrt{2}$ is an algebraic number, if we apply the double angle formula $\cos\theta=\sqrt{\frac{1+\cos2\theta}{2}}$ repeatedly, we see that $\cos\left(\frac{\pi}{2N}\right)$ is an algebraic number too. Thus we obtain the result in Corollary \ref{eventrans}.
\end{proof}
From Corollaries \ref{transc}-\ref{eventrans}, it is clear that the arithmetic nature of $\zeta(2m+1), \zeta\left(2m+1-\frac{1}{N}\right)$ for $m>0$ and $N$ even, and Euler's constant occurring in them is inextricably linked to that of the generalized Lambert series, and hence it may be worthwhile studying the latter from this perspective. That being said, this may be a difficult task. We note that there have been many studies on irrationality of certain Lambert series, for example, by Erd\"{o}s \cite{erdos} and by Luca and Tachiya \cite{lucatachiya}.

\section{A result of Chandrasekharan and Narasimhan}\label{cnesec}

We take this opportunity to correct here an error in one of the results in a paper of Chandrasekharan and Narasimhan \cite{CN}. The corrected version is then seen to be nothing but Ramanujan's formula \eqref{zetaodd} for $m>0$. This corrected version of their result actually appears in two papers of Guinand \cite[Theorem 9]{guinand1944} and \cite[Equation (9)]{guinand} written many years before \cite{CN} appeared, however, the error has not been pointed out in the papers that refer to \cite{CN}. The result of Chandrasekharan and Narasimhan, as given in \cite[Equation (58)]{CN}, is stated below.

\textit{Let $\delta_{m,n}$ denote Kronecker's symbol. For $y>0$, and $k$ and odd integer,
\begin{align}\label{cni}
\sum_{n=1}^{\infty} \sigma_{k}(n) e^{-n y} = \left( \frac{2\pi}{y} \right)^{k+1} \sum_{n=1}^{\infty} (-1)^{\frac{k+1}{2}} \sigma_{k}(n) e^{-\frac{4 \pi^2 n}{y}} + \mathcal{P}(y),
\end{align}
where $\mathcal{P}(y)$ is the sum of the residues of the function $\Gamma(s) \zeta(s) \zeta(s-k) y^{-s}$. If $k>0$, then
\begin{equation*}
\mathcal{P}(y)=\frac{(-1)^{(k-1)/2}B_{(k+1)/2}}{2(k+1)}-\frac{\delta_{1,k}}{2y}+\frac{(2\pi)^{k+1}B_{(k+1)/2}}{2(k+1)y^{k+1}};
\end{equation*}
if $k=-1$, then
\begin{equation*}
\mathcal{P}(y)=\frac{\pi^2}{6y}-\frac{1}{2}\log2\pi+\frac{1}{2}\log y-\frac{y}{24};
\end{equation*}
and if $k <-1$ then 
\begin{equation}\label{cne}
\mathcal{P}(y) = -\frac{1}{2} \zeta(-k).
\end{equation}}
The expression for $\mathcal{P}(y)$ in the case $k>0$ in the above result is correct and, in the cases $k>1$ and $k=1$, it leads to \eqref{ramzetaspl} and \eqref{ramzetaspl0} respectively. The expression corresponding to the case $k=-1$ is also correct. However, we would like draw reader's attention to the fact that the convention for the Bernoulli numbers used by Chandrasekharan and Narasimhan is not the same as the standard notation. Their $B_m$ would be $(-1)^{m+1}B_{2m}$ in the standard notation. (They do not define $B_m$ in their paper, hence the need for clarification.) 

The main thing we would like to emphasize, however, is that the expression in the case $k<-1$, that is, the expression in \eqref{cne}, is incorrect. It is easy to see that if it were correct, it would readily imply that $\zeta(2m+1)$ is irrational for each $m$, a positive integer! This is, as of yet, an unjustified conclusion.

In this section, we obtain the correct expression for $\mathcal{P}(y)$ in \eqref{cne} and thereby obtain a proof of \eqref{zetaodd} for the sake of completeness. However, our exposition will be very brief as the technique of Mellin transforms for deriving such identities is well-known, and since this result is just a special case of Theorem \ref{zetagen} which, in turn, is a special case of Theorem \ref{ktycomp} of which a proof is derived using a similar technique in Section \ref{exten}. Chandrasekharan and Narasimhan derived \eqref{cni} by an application of a general result of Bochner \cite[Lemma 4]{CN}, and Bochner's result is also proved using Mellin transforms.

Let $k=-(2m+1)$ in \eqref{cni}, where $m \in \mathbb{N}$. Using the inverse Mellin transform representation of the exponential function, we see that for $\l=$Re$(s)>1$,
\begin{align}\label{sumcon}
\sum_{n=1}^{\infty} \sigma_{-(2m+1)}(n) e^{-ny}=\frac{1}{2\pi i} \int_{ \lambda - i \infty}^{\lambda + i \infty} \Gamma(s) \zeta(s) \zeta(s+ 2m + 1)  y^{-s}\, {\rm d}s.
\end{align}
Take the contour $\mathcal{C}$ determined by the line segments $[\lambda - i T, \lambda + i T], [\lambda + i T, \mu + i T], [\mu + i T, \mu - i T]$ and $[\mu - i T, \lambda - i T]$, where $-(2m+2)< \mu < - (2m+1)$.
The integrand on the right side of \eqref{sumcon} has simple poles at $s=0, 1, -1, -3, -5, \ldots, -(2m+1)$, and also at $s=-2m$. The error in Chandrasekharan and Narasimhan's paper while deriving \eqref{cne} occurs because they consider $s=0$ to be the only pole of the integrand.

By the Cauchy residue theorem, we have 
\begin{equation}\label{crt}
\begin{split}
& \frac{1}{2\pi i}  \left[\int_{ \lambda - i T }^{\lambda + i T} +  \int_{ \lambda + i T }^{\mu + i T} +
\int_{ \mu + i T }^{\mu - i T} +
 \int_{ \mu - i T }^{\lambda - i T}\right] 
\Gamma(s) \zeta(s) \zeta(s+ 2m + 1)  y^{-s}\, {\rm d}s \\
 &= R_{-2m}+R_{0}+R_{1} + \sum_{i=0}^m R_{-(2i+1)}.
\end{split}
\end{equation}
The above residues are easily calculated to
{\allowdisplaybreaks\begin{align}\label{residues}
R_{-2m}&=(-1)^m \frac{\zeta(2m+1) y^{2m}}{2^{2m+1} \pi^{2m}}, R_0=- \frac{1}{2} \zeta(2m+1), R_1=(-1)^{m} \frac{(2\pi)^{2m +2}B_{2m+2}}{2 (2m+2)! y },\nonumber\\
\sum_{i=0}^m R_{-(2i+1)}&=(-1)^{m+1} \sum_{i=0}^{m} \frac{(-1)^i B_{2i+2} B_{2m-2i} (2\pi)^{2m-2i} y^{2i+1} }{2(2i+2)! (2m-2i)! }.
\end{align}}
As $T\to\infty$, the integrals along the horizontal segments $[\lambda + i T, \mu + i T], [\mu - i T, \l - i T]$ are easily seen to approach zero using Stirling's formula for $\G(s)$ and elementary bounds on the Riemann zeta function. Hence from \eqref{sumcon}, \eqref{crt} and \eqref{residues}, we find that
{\allowdisplaybreaks\begin{align}\label{beffina}
&\sum_{n=1}^{\infty} \sigma_{-(2m+1)}(n) e^{-ny} +   \frac{1}{2\pi i} 
\int_{ \mu + i \infty }^{\mu - i \infty} 
\Gamma(s) \zeta(s) \zeta(s+ 2m + 1)  y^{-s}\, {\rm d}s \nonumber\\ 
&= (-1)^m \frac{\zeta(2m+1) y^{2m}}{2^{2m+1} \pi^{2m}}-\frac{1}{2} \zeta(2m+1)+(-1)^{m} \frac{(2\pi)^{2m +2}B_{2m+2}}{2 (2m+2)! y }\nonumber\\
&\quad+(-1)^{m+1}\sum_{i=0}^{m} \frac{(-1)^i B_{2i+2} B_{2m-2i} (2\pi)^{2m-2i} y^{2i+1} }{2(2i+2)! (2m-2i)! }.
\end{align}}
In the line integral on the left side, which we denote by $V(m)$, we now employ the functional equation \eqref{zetafe} twice to obtain
\begin{equation*}
V(m)=\frac{1}{2\pi i}\int_{ \mu + i\infty }^{\mu - i \infty} (-1)^m (2 \pi)^{2s + 2m}  \Gamma( -s - 2m) \zeta(-s - 2m) \zeta(1-s) y^{-s}\, {\rm d}s,
\end{equation*}
and then make the change of variable $S=-(s+2m)$ so that upon employing again the inverse Mellin transform representation for the exponential function, we have
\begin{align}\label{vm}
V(m)=(-1)^{m+1} \left(\frac{y}{2\pi}\right)^{2m} \sum_{n=1}^{\infty} \sigma_{-(2m+1)}(n) e^{-\frac{4\pi^2 n}{y}}.
\end{align}
Ramanujan's formula \eqref{zetaodd} can now be obtained by first substituting \eqref{vm} in \eqref{beffina}, then letting $y=2\a$ and $\a\b=\pi^2$, multiplying throughout by $(4\a)^{-m}$ and then simplifying the resulting equation.

 \section{Concluding remarks}\label{cr}
We would like to mention that by the principle of analytic continuation all of our results involving $x$, or $\a$ and $\b$ can be extended to their complex values satisfying Re$(x)>0$, Re$(\a)>0$ and Re$(\b)>0$. 

One of the objectives of this paper was to obtain an extension of the beautiful result of Kanemitsu, Tanigawa and Yoshimoto, namely \eqref{ramseries}, for \emph{any} integer value of $h$. 

When $h$ runs through the values $\frac{N+1}{2}+Nm$, where $N$ is an odd positive integer and $m\in\mathbb{Z}\backslash\{0\}$, we get a new elegant generalization of Ramanujan's famous formula for $\zeta(2m+1)$. The nice thing about this generalization is that it gives a relation between \emph{any} two odd zeta values of the form $\zeta(4\ell+3)$. To see this, let $N=2j+1, j\geq 0$ and $m=2\ell+1, l \geq 0$, then Theorem \ref{zetagen} gives a relation between $\zeta(4\ell+3)$ and $\zeta(4(2j\ell+j+\ell)+3)$. If we now let $\ell=0$, then we get a relation between $\zeta(3)$ and $\zeta(4j+3)$. Thus any two distinct values of $j$, for example, $j_1$ and $j_2$, by way of the relations that $\zeta(4j_1+3)$ and $\zeta(4j_2+3)$ have with $\zeta(3)$, imply a relation between $\zeta(4j_1+3)$ and $\zeta(4j_2+3)$. 

However, Theorem \ref{zetagen} gives a similar relation between only some odd zeta values of the form $\zeta(4\ell+1)$, but not between \emph{any} two such odd zeta values. For example, one does not get a relation between $\zeta(5)$ and $\zeta(33)$ from our Theorem \ref{zetagen}. To see this, note that if $N=2j+1, j\geq 1$ and $m=2\ell, \ell \geq 1$, and there exists a relation between $\zeta(2m+1)$ and $\zeta(2Nm+1)$ governed by Theorem \ref{zetagen}, that is, a relation between $\zeta(4\ell+1)$ and $\zeta(4\ell(2j+1)+1)$, then in the special case when $\ell=1$ and $4\ell(2j+1)+1=33$, it would imply $2j+1=8$, which is obviously false for any positive integer $j$. Also, two odd zeta values, one of the form $\zeta(4\ell+1)$ and other of the form $\zeta(4\ell+3)$ do not satisfy the relation in Theorem \ref{zetagen}.

Table 1 below lists some odd zeta values obeying the relation in Theorem \ref{zetagen}.

When $h$ runs through the values $\frac{N}{2}+Nm$, where $N$ is an even positive integer and $m\in\mathbb{Z}$, we obtain a result complementary to Theorem \ref{zetagen}, namely Theorem \ref{neven}. This result covers cases not covered in \cite{ktyhr} and \cite{ktyacta}. For example, if we let $N=2$ and $m=-1$ in Theorem \ref{neven}, we obtain a transformation of the generalized Lambert series $\sum_{n=1}^{\infty}\displaystyle\tfrac{n^{4}}{e^{n^2x}-1}$ which cannot be obtained from the results in \cite{ktyhr} and \cite{ktyacta}. We would also like to emphasize that Theorem \ref{ktycomp} allows us to obtain transformations beyond those given by Theorems \ref{zetagen} and \ref{neven} since $h$ is permitted to take any integral value other than the values $\frac{N+1}{2}+Nm$ and $\frac{N}{2}+Nm$ it takes for Theorems \ref{zetagen} and \ref{neven} respectively.

Corollaries \ref{transc}, \ref{gammatrans} and \ref{eventrans} suggest that the key to studying questions on transcendence of $\gamma$ and odd zeta values might lie in the study of the generalized Lambert series occurring in these results for $N>1$. Note that Corollary \ref{gammatrans} gives information about Euler's constant $\gamma$ only when $N>1$ since for $N=1$, the expression involving $\gamma$ in Theorem \ref{zetagenm0} vanishes. Similarly when $m$ is an even natural number, the result in Corollary \ref{transc}, as well as that in the second part of Corollary \ref{rtype}, holds only when $N$ is an odd integer greater than $1$ since the $N=1$ case of Theorem \ref{zetagen}, which is nothing but \eqref{zetaodd}, simply reduces to $\sum_{j=0}^{m+1}\frac{(-1)^jB_{2j}B_{2m+2-2j}}{(2j)!(2m+2-2j)!}=0$ 
when $m$ is an even natural number.

Ramanujan's formula \eqref{zetaodd} can be interpreted \cite{gmr}, \cite{berndtstraubzeta} as the formula encoding the fundamental transformation properties of Eisenstein series of level $1$ and their Eichler integrals. In light of this, it is important to see what is encoded by Theorem \ref{zetagen}, a generalization of Ramanujan's formula. Moreover, Theorem \ref{neven} being a natural complement of Theorem \ref{zetagen}, suggests a similar study when $N$ is even.

Our results in Theorems \ref{zetagen} and \ref{neven}, and more generally in Theorem \ref{ktycomp}, can be generalized in the context of Hurwitz zeta function and Dirichlet $L$-functions. In fact, as mentioned before, Kanemitsu, Tanigawa and Yoshimoto have already obtained a generalization of Theorem \ref{ktycomp} for Hurwitz zeta function \cite{ktyacta}, and later for Dirichlet $L$-function \cite{ktypic}, but they only consider the case when $N$ is even and $h\geq N/2$. Thus, these results could be further generalized when $N$ is any positive integer and $h$ is any integer. 

Lastly we would like to mention that the behavior of the Lambert series $\displaystyle\sum_{n=1}^{\infty}\frac{n^{N-(2h+1)}}{e^{n^{N}x}-1}$ seems to be quite different from the one studied here. This will be studied in a forthcoming paper \cite{dixitmaji2}.

\begin{table}[!htb]
\caption{\bf Odd zeta values related by Theorem \ref{zetagen}}
\begin{tabular}{|c|c|c|c|}
\hline
$m$ & $N$ & $\zeta(2m+1)$ & $\zeta(2Nm +1)$ \\ 
\hline
1 & 1 & $\zeta(3)$ & $\zeta(3)$  \\
\hline
1 & 3 & $\zeta(3)$ & $\zeta(7)$ \\
\hline
1 & 5 & $\zeta(3)$ & $\zeta(11)$ \\
\hline
1 & 7 & $\zeta(3)$ & $\zeta(15)$ \\
\hline
\end{tabular}
\quad\hspace{1cm}
\begin{tabular}{|c|c|c|c|}
\hline
$m$ & $N$ & $\zeta(2m+1)$ & $\zeta(2Nm +1)$ \\ 
\hline
2 & 3 & $\zeta(5)$ & $\zeta(13)$ \\
\hline
2 & 5 & $\zeta(5)$ & $\zeta(21)$ \\
\hline
2 & 7 & $\zeta(5)$ & $\zeta(29)$ \\
\hline
2 & 9 & $\zeta(5)$ & $\zeta(37)$ \\
\hline
\end{tabular}
\end{table}

\begin{table}[!htb]
\begin{tabular}{|c|c|c|c|}
\hline
$m$ & $N$ & $\zeta(2m+1)$ & $\zeta(2Nm +1)$ \\ 
\hline
3 & 1 & $\zeta(7)$ & $\zeta(7)$ \\
\hline
3 & 3 & $\zeta(7)$ & $\zeta(19)$ \\
\hline
3 & 5 & $\zeta(7)$ & $\zeta(31)$ \\
\hline
3 & 7 & $\zeta(7)$ & $\zeta(43)$ \\
\hline
\end{tabular}
\quad\hspace{1cm}
\begin{tabular}{|c|c|c|c|}
\hline
$m$ & $N$ & $\zeta(2m+1)$ & $\zeta(2Nm +1)$ \\ 
\hline
4 & 3 & $\zeta(9)$ & $\zeta(25)$ \\
\hline
4 & 5 & $\zeta(9)$ & $\zeta(41)$ \\
\hline
4 & 7 & $\zeta(9)$ & $\zeta(57)$ \\
\hline
4 & 9 & $\zeta(9)$ & $\zeta(73)$ \\
\hline
\end{tabular}
\end{table}
\vspace{17mm}
\begin{center}
\textbf{Acknowledgements}
\end{center}
The authors thank Professor Bruce C. Berndt for his encouragement and for giving important suggestions which improved the exposition. Part of this work was written while the first author was visiting Harish-Chandra Research Institute during June 27-July 4, 2017. He sincerely thanks the institute for the hospitality and support. The first author's research is supported by the SERB-DST grant RES/SERB/MA/P0213/1617/0021 and he sincerely thanks SERB-DST for the support. The second author is a postdoctoral fellow at Indian Institute of Technology Gandhinagar.

\end{document}